\documentclass[preprint,12pt]{elsarticle}

\usepackage{amssymb}

\usepackage{tikz,url}
\usetikzlibrary{positioning}
\usepackage{amsmath}
\usepackage{amsthm}
\usepackage{caption}
\usepackage{subcaption}
\usepackage[ruled,vlined]{algorithm2e}

\newtheorem{theorem}{Theorem}[section]
\newtheorem{question}{Question}
\newtheorem{lemma}{Lemma}[section]

\newtheorem{obs}{Observation}[section]
\newtheorem{claim}{Claim}[section]

\usepackage{graphicx}
  \baselineskip 8 mm

\begin{document}
\begin{frontmatter}


\title{Characterizing optimal monitoring edge-geodetic sets for some structured graph classes}

\author[inst2]{Florent Foucaud\footnote{Research of Florent Foucaud partially funded by the IDEX-ISITE initiative CAP 20-25 (ANR-16-IDEX-0001), the International Research Center ``Innovation Transportation and Production Systems'' of the I-SITE CAP 20-25, and the ANR project GRALMECO (ANR-21-CE48-0004)}}
\ead{florent.foucaud@uca.fr}

\author[inst1]{Arti Pandey\footnote{Research of Arti Pandey is supported by CRG project, Grant Number-CRG/2022/008333, Science
and Engineering Research Board (SERB), India}}
\ead{arti@iitrpr.ac.in}

\author[inst1]{Kaustav Paul\corref{mycorrespondingauthor}}
\cortext[mycorrespondingauthor]{Corresponding author}
\ead{kaustav.20maz0010@iitrpr.ac.in}

\affiliation[inst2]{
            addressline={Université Clermont Auvergne, CNRS, Clermont Auvergne INP, Mines Saint-Étienne, LIMOS}, 
            city={63000, Clermont-Ferrand},
            country={France}}
\affiliation[inst1]{
            addressline={Department of Mathematics, Indian Institute of Technology Ropar}, 
            city={Rupnagar},
            postcode={140001}, 
            state={Punjab},
            country={India}}






\begin{abstract}
    Given a graph $G=(V,E)$, a set $S\subseteq V$ is said to be a monitoring edge-geodetic set if the deletion of any edge in the graph results in a change in the distance between at least one pair of vertices in $S$. The minimum size of such a set in $G$ is called the monitoring edge-geodetic number of $G$ and is denoted by $meg(G)$.

    In this work, we compute the monitoring edge-geodetic number efficiently for the following graph classes: distance-hereditary graphs, $P_4$-sparse graphs, bipartite permutation graphs, and strongly chordal graphs. The algorithms follow from structural characterizations of the optimal monitoring edge-geodetic sets for these graph classes in terms of \emph{mandatory vertices} (those that need to be in every solution). This extends previous results from the literature for cographs, interval graphs and block graphs.
\end{abstract}

\end{frontmatter}

\section{Introduction}
The realm of network monitoring entails a range of applications focused on detecting and addressing periodic failures within a network. Within this framework, a network is considered a finite simple graph, where failures may manifest as connection losses between two nodes of the network. Such failures often lead to temporary increases in distances or the isolation of specific nodes from the broader network. Ideally, consistent monitoring of the entire graph would promptly reveal any instances of malfunction. Nevertheless, developing a methodology that achieves this objective with minimized cost is more advantageous.

In a recent study, Foucaud et al. \cite{DBLP:conf/caldam/FoucaudNS23} introduced a new graph parameter based on the distances between a specific group of vertices, motivated by applications in network monitoring. This concept arises from merging the classic concept of \emph{geodetic sets}~\cite{harary1993geodetic} and the more recent one of \emph{distance edge-monitoring sets}~\cite{foucaud2022monitoring}. An edge of a graph $G=(V,E)$ is said to be \emph{monitored} by a pair of vertices $a,b\in V$, if every shortest $a-b$ path contains $e$. An edge $e$ of a graph $G=(V,E)$ is said to be \emph{monitored} by a set $S\subseteq V$, if $e$ is monitored by at least one pair of vertices in $S$. A monitoring edge-geodetic set of $G$, abbreviated as MEG set of $G$, refers to a subset of the vertex set $V$, that monitors all the edges of $G$. The objective is to identify an MEG set that possesses the smallest possible size. The \emph{monitoring edge-geodetic number} (or MEG number) of $G$, denoted as $meg(G)$, represents the minimum size of an MEG set of a graph $G$. The optimization and decision versions of the problem are defined as follows:

\medskip

\noindent\underline{\textsc{Minimum Monitoring Edge-Geodetic Set} problem (MIN-MEG)}
\begin{enumerate}
  \item[] \textbf{Instance}: A graph $G=(V,E)$.
  \item[] \textbf{Solution}: An MEG set $M$ of $G$ with $\vert M\vert=meg(G)$.
\end{enumerate}

Our goal is to continue the study of these computational problems on structured graph classes. To do so, we will use the following concept: Given a graph $G=(V,E)$, we define the notion of \emph{mandatory vertex} as follows: a vertex $v\in V$ is \emph{mandatory} in $G$ if it is contained in every MEG set of $G$. We also define the set $Man(G)=\{v\in V~\vert~v$ is mandatory in $G\}$.

\subsection{Related work}
The monitoring edge-geodetic set problem was introduced by Foucaud et al. \cite{DBLP:conf/caldam/FoucaudNS23} in 2023, see also~\cite{DBLP:journals/corr/abs-2210.03774}. In their work, they computed the monitoring edge-geodetic number for some fundamental graph classes such as trees, cycles, unicyclic graphs, complete graphs, grids, and hypercubes. Additionally, using the feedback edge set number of a graph $G$, they derived an upper bound for $meg(G)$, which was subsequently refined in~\cite{DBLP:conf/fct/ChakrabortyFH23}. In the same paper, the authors presented several examples of graphs that satisfy $meg(G)=n$ (where $n$ is the order of $G$), and asked whether it is possible to characterize the class of graphs for which $meg(G)=n$. The first NP-hardness proof for the MIN-MEG problem for general graphs was proposed by Haslegrave \cite{DBLP:journals/dam/Haslegrave23}. In that paper, the author also produced some bounds on the MEG number for the certain product operations of two graphs, $G$ and $H$, in terms of $meg(G)$ and $meg(H)$. See also~\cite{DBLP:journals/paapp/XuYBZS24} for further work on MEG on graph products, and~\cite{doi:10.1142/S0219265924500105,DBLP:journals/paapp/MaJYL24,tan2023monitoring,doi:10.1142/S0219265924500014} for MEG on various graph families. In \cite{DBLP:journals/corr/abs-2403-09122,DBLP:conf/caldam/FoucaudMMSST24}, Foucaud et al. characterized the graphs of order $n$ for which $meg(G)=n$, answering the question asked in \cite{DBLP:journals/corr/abs-2210.03774,DBLP:conf/caldam/FoucaudNS23}, by characterizing mandatory vertices. Indeed, when $meg(G)=n$, all vertices must be mandatory. They also provided an upper bound for the MEG number in terms of girth for sparse graphs, and showed that the MIN-MEG problem is NP-hard for the family of $3$-degenerate, $2$-apex graphs, improving the existing result of \cite{DBLP:journals/dam/Haslegrave23}. Recently, Bil{\`{o}} et al, have shown that no polynomial time $(c\log{n})$ (where $c<\frac{1}{2}$) factor approximation algorithm exists for the MIN-MEG problem, unless $P=NP$ \cite{DBLP:conf/ictcs/BiloCF024}. In \cite{foucaud2024algorithmscomplexitymonitoringedgegeodetic}, Foucaud et al. have shown that the problem cannot be solved in subexponential time even for $3$-degenerate graphs unless the Exponential time hypothesis (ETH) fails. They also show that the MIN-MEG problem is APX-hard for $4$-degenerate graphs. Complementing these hardness results, they also prove that the problem admits a polynomial-time algorithm for interval graphs, a FPT algorithm for general graphs with clique-width plus diameter as the parameter, and a FPT algorithm for chordal graphs with treewidth as the parameter. They also provide an approximation algorithm with factor $\ln{m} \cdot meg(G)$ and $\sqrt{n \ln{m}}$ for the MIN-MEG problem (where $n$ and $m$ denote the number of vertices and edges of the input graph).

\subsection{Our results}
By Theorem \ref{th:cut}, we can conclude that the set of cut-vertices is not part of any minimum MEG set of $G$. Hence, it follows that no minimum-size MEG set of a graph $G$ includes any cut-vertex.   We define $Cut(G)=\{v\in V~\vert~v$ is a cut-vertex of $G\}$. Then, for any minimum-size MEG set $M$, we have $M\subseteq V\setminus Cut(G)$. The undefined terms in this section are defined in Section $2$. As shown in~\cite{DBLP:journals/corr/abs-2403-09122,DBLP:conf/caldam/FoucaudMMSST24}, sometimes the (unique) minimum-size MEG set of $G$ is $V\setminus Cut(G)$. Hence, the following observation holds:

\begin{obs}\label{observation:2}
    Given a graph $G=(V,E)$, let $M$ be a minimum-size MEG set of $G$. Then, $Man(G)\subseteq M\subseteq V\setminus Cut(G)$. Hence, $\vert Man(G)\vert\leq meg(G)\leq \vert V\setminus Cut(G)\vert$.
\end{obs}

\medskip
\noindent By Observation~\ref{observation:2}, we can conclude that several graph classes can be classified as one of the following two types:

\begin{enumerate}
    \item the graphs for which $Man(G)$ is an optimal MEG set
    \item even stronger, the graphs for which $Man(G)=V\setminus Cut(G)$.
\end{enumerate}

\noindent Note that for a given graph $G=(V,E)$, if $G$ is a graph of type $1$, then $meg(G)$ can be computed in polynomial time, since the necessary and sufficient condition for a vertex to be a mandatory vertex, mentioned in Theorem~\ref{th:mandatory}, can be checked in polynomial time. If $G$ is of type $2$, then we have an even simpler linear time algorithm to compute $meg(G)$. Hence, if $meg(G)$ attains either of the upper and lower bound mentioned in Observation~\ref{observation:2}, it can be computed in polynomial time as well. It was shown in~\cite{foucaud2024algorithmscomplexitymonitoringedgegeodetic} that interval graphs are of type 1, and in~\cite{DBLP:journals/corr/abs-2403-09122,DBLP:conf/caldam/FoucaudMMSST24} that block graphs, proper interval graphs, cographs and well-partitioned chordal graphs are of type 2.

We extend and strengthen some of these results by considering more general graph classes, as follows. 

\begin{itemize}
    \item In Section \ref{sec:distance-hereditary}, we show that for any distance-hereditary graph $G$, $Man(G)=V\setminus Cut(G)$, which implies that $meg(G)=\vert Man(G)\vert=\vert V\setminus Cut(G)\vert$.
    \item In Section \ref{sec:P4sparse}, we show that for any $P_4$-sparse graph $G$, $meg(G) = \vert Man(G)\vert$.
    \item In Section \ref{sec:bip_per}, we show that for any bipartite permutation graph $G$, $Man(G)=V\setminus Cut(G)$, which implies that $meg(G)=\vert Man(G)\vert=\vert V\setminus Cut(G)\vert$.
    \item In Section \ref{sec:strong_chordal}, we show that for any strongly chordal graph $G$, $meg(G)=\vert Man(G)\vert$. 
    \item Section \ref{sec:conclusion} concludes our work.

\end{itemize}
Hence, for all the above graph classes, the monitoring edge-geodetic number can be computed in polynomial time. In Figure~\ref{fig:diagram}, we show a hierarchy tree of graph classes to signify the improvements achieved by our results over the existing literature.

\begin{figure}[h]
\centering
\scalebox{0.7}{\begin{tikzpicture}[node distance=7mm]

\tikzstyle{mybox}=[fill=white,line width=0.5mm,rectangle, minimum height=.8cm,fill=white!70,rounded corners=1mm,draw];
\tikzstyle{myedge}=[line width=0.5mm]
\newcommand{\tworows}[2]{\begin{tabular}{c}{#1}\\{#2}\end{tabular}}
\tikzstyle{ours}=[line width=1mm]
\tikzstyle{man}=[fill=gray!20]
\tikzstyle{cut}=[fill=gray!50]

    \node[mybox] (chordal) {chordal};
    \node[mybox] (cocomp) [right =of chordal] {co-comparability};
    \node[mybox] (circle) [right =of cocomp, xshift=23mm] {circle};
    \node[mybox, man, ours] (strongly) [below =of chordal,yshift=3mm] {strongly chordal} edge[myedge] (chordal);
    \node[mybox, cut] (wpchordal) [left =of strongly,xshift=-3mm] {well-partitioned chordal \cite{DBLP:journals/corr/abs-2403-09122,DBLP:conf/caldam/FoucaudMMSST24}} edge[myedge] (chordal);
    \node[mybox, cut] (split) [below =of wpchordal] {split  \cite{DBLP:journals/corr/abs-2403-09122,DBLP:conf/caldam/FoucaudMMSST24}} edge[myedge] (wpchordal);
    \node[mybox, man] (interval) [below =of cocomp,yshift=-13mm] {interval  \cite{foucaud2024algorithmscomplexitymonitoringedgegeodetic}} edge[myedge] (strongly) edge[myedge] (cocomp);
    \node[mybox] (perm) [below right=of cocomp,yshift=1mm, xshift=5mm] {permutation} edge[myedge] (cocomp) edge[myedge] (circle);
    \node[mybox, cut] (proper) [below=of interval,yshift=3mm] {proper interval \cite{DBLP:journals/corr/abs-2403-09122,DBLP:conf/caldam/FoucaudMMSST24}} edge[myedge] (interval);
    \node[mybox, cut, ours] (bperm) [below =of perm] {bipartite permutation} edge[myedge] (perm);
    \node[mybox, cut, ours] (dh) [right=of perm] {distance-hereditary} edge[myedge] (circle);
    \node[mybox, man, ours] (P4) [right=of dh] {$P_4$-sparse};
    \node[mybox, cut] (cograph) [right=of bperm] {cograph \cite{DBLP:journals/corr/abs-2403-09122,DBLP:conf/caldam/FoucaudMMSST24}} edge[myedge] (perm)  edge[myedge] (dh) edge[myedge] (P4);
    \node[mybox, cut] (block) [below =of strongly,yshift=1mm]
    {block \cite{DBLP:journals/corr/abs-2403-09122,DBLP:conf/caldam/FoucaudMMSST24}} edge[myedge] (strongly) edge[myedge] (wpchordal);
    \node[mybox, cut] (tree) [below=of block,yshift=3mm] {trees \cite{DBLP:conf/caldam/FoucaudNS23}} edge[myedge] (block);
  \end{tikzpicture}}

\caption{Inclusion diagram for graph classes mentioned in this paper (and related ones). If a class $A$ has an upward path to class $B$, then $A$ is included in $B$. For any graph $G$ in a dark gray class, we have $\vert Man(G)\vert=\vert V(G)\setminus Cut(G)\vert=meg(G)$. For any graph $G$ in a light gray class, we have $\vert Man(G)\vert=meg(G)$. Results for boxes with a thick border are proved in this paper.}
\label{fig:diagram}
\end{figure}
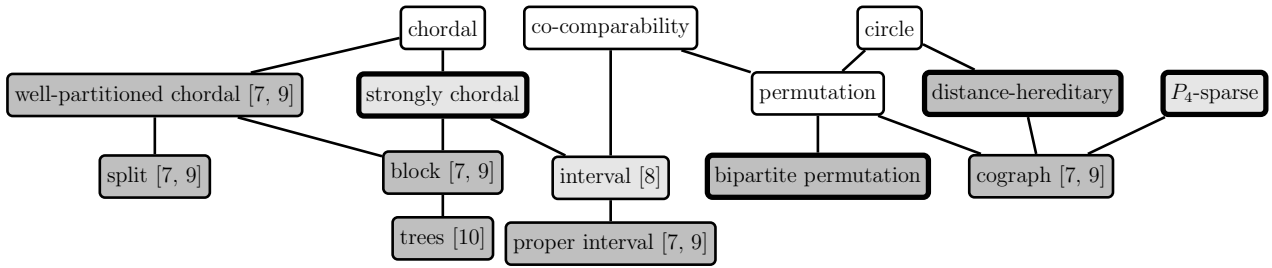

\section{Preliminaries}

\subsection{Notations and definitions}
This paper only considers simple, undirected, finite, and connected graphs with at least two vertices. Let $G=(V,E)$ be a graph; $n$ and $m$ are used to denote the cardinalities of $V$ and $E$ respectively. $N(v)$ stands for the set of neighbors of a vertex $v$ in $V(G)$. The number of neighbors of a vertex $v\in V$ defines its \emph{degree}, which is represented by the symbol $deg(v)$. The maximum degree of the graph $G$ is denoted by $\Delta(G)$; when the context is clear, we use only $\Delta$. Given a subset of vertices $U$ of $V$, the notation $deg_{U}(v)$ is used to represent the number of neighbors that a vertex $v$ has within the set $U$. Additionally, we use $N_{U}(v)$ to refer to the set of neighbors of vertex $v$ within $U$. A vertex $v\in V$ is said to be an \emph{pendant vertex} (resp. \emph{isolated vertex}) if $\vert N(v)\vert=1$ (resp. $\vert N(v)\vert=0$). The unique neighbour of a pendant vertex is called a \emph{stem}. A vertex $v\in V$ is said to be an \emph{isolated vertex} if $N(v)=\emptyset$. An edge $uv\in E$ is said to be an \emph{isolated edge} if $deg(u)=deg(v)=1$. A vertex $v$ is said to be a \emph{simplicial} vertex in $G$ if $N[v]$ induces a clique on $G$. 

A path $P = v_0v_1\ldots v_k$, is a sequence of distinct
vertices, such that $v_{i-1}v_i \in E$, where $1 \leq  i \leq k$ and $k \geq 2$. Such a path is called a path between $v_0$ and $v_k$. We denote $V(P) = \{v_0, v_1,\ldots, v_k\}$. The length of the path $P$ is $\vert V(P)\vert - 1$. Given a path $P$ and two vertices $x,y$ of $P$, the distance of $x$ and $y$ in $P$ is defined as $d_P(x,y)$ and the induced path $P$ on $x,y$ and all the intermediate vertices between $x$ and $y$ in $P$ is denoted as $P(x,y)$.

The \emph{join} of two graphs $G_{1}$ and $G_{2}$ refers to a graph formed by taking separate copies of $G_{1}$ and $G_{2}$ and adding edges joining every vertex in $V(G_{1})$ to each vertex in $V(G_{2})$. The symbol $\oplus$ denotes the join operation. Similarly, \emph{disjoint union} of two graphs $H_1$ and $H_2$ is the graph $H=(V(H_1)\cup V(H_2), E(H_1)\cup E(H_2))$. The disjoint union is denoted with the symbol $\cup$.

Two vertices $u$ and $v$ are called \emph{true twins} if $N_G[u] = N_G[v]$.
Similarly, $u$ and $v$ are called \emph{false twins} if $N_G(u) = N_G(v)$. Vertices $u$ and $v$ are called \emph{twins} if they are either true or false twins.

A vertex $v$ of a graph $G=(V,E)$ is said to be a \emph{universal vertex} if $N[v]=V$. A pair of nonadjacent vertices $\{u,u'\}$ of $G$ is said to be a \emph{universal pair} if every $x\in V\setminus\{u,u'\}$ is adjacent to both $u$ and $u'$. A set $A\subseteq V(G)$ is known to be a \emph{stable set} or an \emph{independent set} of $G$ if the graph induced on $A$, that is $G[A]$ contains no edge.

A graph $G=(V,E)$ is said to be $P_4$-sparse if the graph induced on any five vertices contains at most one $P_4$.  A \emph{spider} is a graph $G=(V,E)$, where $V$ admits a partition into three subsets $S,C$ and $R$ such that
\begin{itemize}
     \item $C=\{c_1,\ldots,c_l\}~(l\geq 2)$ is a clique.
    \item $S=\{s_1,\ldots,s_l\}$ is a stable set.
    \item Every vertex in $R$ is adjacent to every vertex in $C$ and nonadjacent to every vertex of $S$. The edges in $R$ are unrestricted.
\end{itemize}
More on $P_4$-sparse graphs and spider graphs can be found in \cite{DBLP:journals/dam/JamisonO92}. A spider $G(S,C,R)$ is said to be a 
\begin{itemize}
    \item \emph{thin spider} if for every $i\in \{1,\ldots,l\}$, $N_{C}(s_i)=\{c_i\}$
    \item \emph{thick spider} if for every $i\in \{1,\ldots,l\}$, $N_{C}(s_i)=C\setminus\{c_i\}$.
\end{itemize}

\begin{figure}[h!]
    \centering
    \begin{subfigure}{0.4\textwidth}
        \centering
        \includegraphics[width=0.5\linewidth]{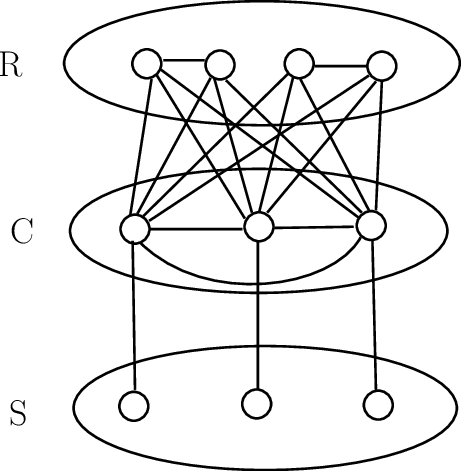}
    \caption{Thin spider}
    \label{subfig:1}
    \end{subfigure}
\hspace{0.5in}
    \begin{subfigure}{0.4\textwidth}
    \centering
        \includegraphics[width=0.5\linewidth]{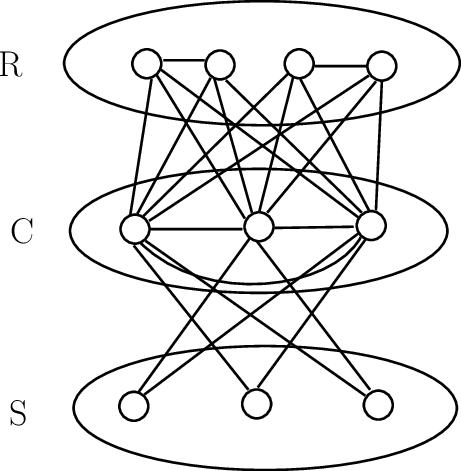}
    \caption{Thick spider}
    \label{subfig:2}
    \end{subfigure}
    \caption{Examples of spiders with spider partition $(S,C,R)$}
\label{fig:spi}
\end{figure}

A graph $G=(V,E)$ is said to be \emph{bipartite} if $V(G)$ can be partitioned into two independent sets $X$ and $Y$. A graph is a \emph{permutation graph} if and only if it has an intersection model consisting of straight lines (one per vertex) between two parallel lines. A graph that is both a bipartite and permutation graph is called a \emph{bipartite permutation} graph.

A graph $G=(V,E)$ is said to be \emph{chordal}, if there does not exist any induced cycle of length at least $4$ in $G$.  $G$ is \emph{strongly chordal} if it is a chordal graph and every cycle of even length (at least $6$) in $G$ has an odd chord, that is, an edge that connects two vertices which are an odd distance ($>1$) apart from each other in the cycle.

\emph{Distance-hereditary} graphs are graphs in which the distance between any two vertices in any connected induced subgraph is the same as their distance in the original graph.

\subsection{Results from previous literature}
The following results are from the existing literature on the problem and are used in the upcoming proofs.

\begin{theorem}[\cite{DBLP:conf/caldam/FoucaudMMSST24}]\label{th:mandatory}
    Given a graph $G=(V,E)$ and a vertex $v\in V$, $v$ is a mandatory vertex of $G$ if and only if there exists $u\in N(v)$ such that every induced $2$-path $uvx$ is part of a $4$-cycle.
\end{theorem}

\noindent We also propose the following observation, which follows directly from Theorem~\ref{th:mandatory}.

\begin{obs}\label{observation:1}
    Let $u,v\in V$, if $N_G(u)\subseteq N_G(v)$ or $N_G[u]\subseteq N_G[v]$, then $u\in M$, for every MEG set $M$ of $G$. In particular, every simplicial vertex is a mandatory vertex.
\end{obs}

\noindent The next theorem is a corollary of the Theorem \ref{th:mandatory}.

\begin{theorem}[\cite{DBLP:journals/corr/abs-2210.03774,DBLP:conf/caldam/FoucaudNS23}]\label{th:man_twin}
    Given a graph $G=(V,E)$, if two vertices $u$ and $v$ are twins 
    of degree at least $1$ in a graph $G$, then they both are mandatory vertices of $G$.
\end{theorem}

\begin{theorem}[\cite{DBLP:conf/caldam/FoucaudMMSST24}]\label{th:man_pendant}
    Let $G=(V,E)$ be a graph with a path $v_0v_1\ldots v_{k-1}v_k$ whose internal vertices have degree $2$, $v_0$ has degree at least $2$, and $v_k$ has degree $1$. Then the vertices $v_0, v_1,\ldots, v_{k-1}$ are never part of any minimum-size MEG set.
\end{theorem}

\begin{theorem}[\cite{DBLP:journals/corr/abs-2210.03774,DBLP:conf/caldam/FoucaudNS23}]\label{th:cut}
    Let $G$ be a graph with a cut-vertex $v$ and $C_1, C_2,\ldots , C_k$ be the $k$ components obtained when removing $v$ from $G$. If $S_1,S_2,\ldots, S_k$ are MEG sets of the induced subgraphs $G[C_1 \cup \{v\}],~G[C_2 \cup \{v\}],\ldots, G[C_k \cup \{v\}]$, then $S=(S_1 \cup S_2\cup \ldots S_k) \setminus \{v\}$ is an MEG set of G.
\end{theorem}

\section{MEG for distance-hereditary graphs}
\label{sec:distance-hereditary}
In this section, we show that the MIN-MEG problem can be solved efficiently for distance-hereditary graphs. Throughout this section, we assume that $G=(V,E)$ is a connected distance-hereditary graph.

For any given ordering
$(v_1, v_2,\ldots,v_n)$ of $V$, let $G_i$ be the graph induced on $\{v_1, v_2,\ldots,v_i\}$. A \emph{one vertex-extension ordering} of G is an ordering $(v_1, v_2,\ldots,v_n)$ of $V$ such that $v_i$ is a leaf or a twin (true or false) of some vertex in $G_i$ for each $i$, $2 \leq i \leq n$. Hammer et al. \cite{DBLP:journals/dam/HammerM90} proved that a graph is distance-hereditary if and only if it has a one-vertex-extension
ordering. We show that for a distance-hereditary graph $G$, $V\setminus Cut(G)$ forms a minimum MEG set of $G$.

\begin{theorem}
    Given a connected distance-hereditary graph $G = (V, E)$, $Man(G)=V \setminus Cut(G)$ and thus, this set is a minimum MEG set of $G$, that can be computed in polynomial time.
\end{theorem}

\begin{proof}
    Our objective is to prove that every vertex $v\in V\setminus Cut(G)$ is mandatory in $G$. To prove this, we employ induction on $\vert V\vert$. The base case, $\vert V\vert=3$, holds true, as the only connected distance-hereditary graphs of order $3$ are $P_3$ and $C_3$, for both of them, $V\setminus Cut(G)= Man(G)$. Now let us assume that every $v\in V\setminus Cut(G)$ is contained in $Man(G)$, whenever $\vert V\vert\leq k$. 
    
    Let us consider a connected distance-hereditary graph $G=(V,E)$ with $\vert V\vert=k+1$, adhering to a one-vertex extension ordering $(v_1,v_2,\ldots,v_k,v_{k+1})$. As per the definition of a one-vertex extension ordering, $G$ is constructed from $G_k$ by appending $v_{k+1}$ as either a pendant or a twin (true or false) to some $v_j$, where $j\leq k$. We inspect two potential scenarios:

    \noindent\textbf{Case 1}: \textbf{$v_{k+1}$ is a true twin (or false twin) of $v_j$} (refer to Figure \ref{fig:true_twin}). 
    
    \begin{figure}[h!]
        \centering
        \includegraphics[scale=0.5]{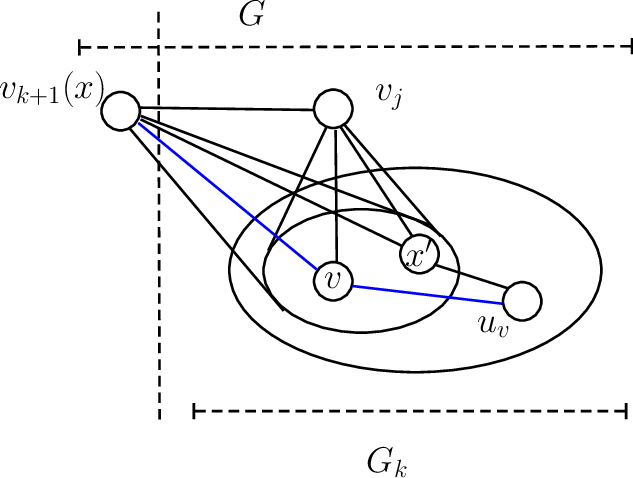}
        \caption{The mandatory vertices in $V(G_k)\setminus Cut(G_k)$ remain mandatory in $V(G)\setminus Cut(G)$}
        \label{fig:true_twin}
    \end{figure}
    
    \noindent Note that $Cut(G)=Cut(G_k)\setminus \{v_j\}$ and $V(G)\setminus Cut(G)=(V(G_k)\setminus Cut(G_k))\cup\{v_j,v_{k+1}\}$. We show that the set $V(G)\setminus Cut(G)$ is contained in every MEG set of $G$. Since $v_j,v_{k+1}$ are twins, they are contained in every MEG set of $G$, by Theorem \ref{th:man_twin}. Hence, $v_j,v_{k+1}\in Man(G)$. 

    \smallskip
    \noindent By the induction hypothesis, the set $V(G_k)\setminus Cut(G_k)$ is contained in every MEG set of $G_k$. This implies that for any $v\in V(G_k)\setminus Cut(G_k)$, there exists $u_v \in N_{G_k}(v)$ such that any induced $2$-path $u_vvx$ is part of a $4$-cycle in $G_k$. Next, we show that $v\in Man(G)$. For the sake of contradiction, assume $v\notin Man(G)$; this implies that there exists an induced $2$-path $u_vvx$ that is part of no $4$-cycle in $G$. But this implies that this $2$-path was not present in $G_k$ (the path highlighted in blue in Figure \ref{fig:true_twin}). Hence, $x=v_{k+1}$. This implies that $v\in N(v_{k+1})$ and $u_v\in N(v)\setminus N(v_{k+1})$; hence $v\in N(v_{k+1})\setminus\{v_j\}$. But note that $v$ was mandatory in $G_k$, which implies $u_vvv_j$ is part of a $4$-cycle $u_vvv_jx'$ in $G_k$ (refer to Figure \ref{fig:true_twin}). This implies $u_vvv_{k+1}$ is also part of the $4$-cycle $u_vvv_{k+1}x'$ in $G$, leading to a contradiction. Hence, $v\in Man(G)$.

     \smallskip
    \noindent Hence, following the addition of the true twin $v_{k+1}$ with $v_j$ in $G_{k}$, the vertices in $V(G_k)\setminus Cut(G_k)$, which possess the aforementioned property in $G_k$, retain said property in $G$. Hence, every vertex of $V(G)\setminus Cut(G)$ is contained in  $Man(G)$. Note that the proof of the case when $v_{k+1}$ is a false twin of $v_j$, is exactly analogous, hence omitted.

    \medskip
    \noindent\textbf{Case 2}: \textbf{$v_{k+1}$ is a pendant vertex adjacent to $v_j$}.

    \smallskip
    \noindent Note that $Cut(G)=Cut(G_k)\cup \{v_j\}$. Since $v_{k+1}$ is a pendant vertex, it is contained in $Man(G)$, by Theorem \ref{th:man_pendant}. By the induction hypothesis, $V(G_k)\setminus Cut(G_k)$ is contained in every MEG set of $G_k$. This implies that for every $v\in V(G_k)\setminus (Cut(G_k)\cup \{v_j\})$, there exists $u_v \in N(v)$ such that any induced $2$-path $u_vvx$ is part of a $4$-cycle. Hence, after the pendant addition of $v_{k+1}$ with $v_j$, the vertices in $V(G_k)\setminus (Cut(G_k)\cup \{v_j\})$, which holds the above property in $G_k$ also hold the property in $G$. This implies that $V(G)\setminus Cut(G)$ is contained in every MEG set of $G$.

    \medskip
    \noindent Hence, by induction, we have shown that $V\setminus Cut(G) \subseteq Man(G)$. Hence, $V\setminus Cut(G) = Man(G)$ and $V\setminus Cut(G)$ is a minimum MEG set of $G$.
\end{proof}

\section{MEG for $P_4$-sparse graphs}\label{sec:P4sparse}
In this section, we propose an efficient algorithm that solves the MIN-MEG problem for $P_4$-sparse graphs. The class of $P_4$-sparse graphs is an extension of the class of cographs. Below, we state a characterization theorem for $P_4$-sparse graphs from the previous literature.

\begin{theorem}[\cite{Defn_P4sparse}]\label{th:5}
    A graph $G$ is said to be $P_4$-sparse if and only if one of the following conditions hold
    \begin{itemize}
        \item $G$ is a single vertex graph.
        \item $G=G_1\cup G_2$, where $G_1$ and $G_2$ are $P_4$-sparse graphs.
        \item $G=G_1 \oplus G_2$, where $G_1$ and $G_2$ are $P_4$-sparse graphs.
        \item $G$ is a spider which admits a spider partition $(S, C, R)$ where either $G[R]$ is a $P_4$-sparse graph or $R=\emptyset$.
    \end{itemize} 
\end{theorem}

Hence, by Theorem \ref{th:5}, a graph that is $P_4$-sparse and contains at least two vertices can be classified as either a join or union of two $P_4$-sparse graphs or a particular type of spider (thick or thin). Consequently, in this section, we compute $meg(G)$ for each of the cases. First, we consider the case of $G$ being a spider.

\begin{lemma}\label{lem:thin}
    Let $G=(V,E)$ be a thin spider and $(S,C,R)$ be the spider partition. Then, $Man(G)=S\cup R$, which is an optimal MEG set of $G$.
\end{lemma}

\begin{proof}
    Note that by Observation \ref{observation:1}, $S$ is contained in every MEG set (as every vertex of $S$ is a pendant vertex) and no vertex in $C$ can be part of a minimum MEG set, as they are cut-vertices. Now if $R\neq \emptyset$, then for each $r\in R$ and for any $c\in C$, $N[r]\subset N[c]$, hence by Observation \ref{observation:1}, $r$ is contained in every MEG set. Hence $meg(G)=\vert Man(G)\vert=\vert S\vert+\vert R\vert$.
\end{proof}

\begin{lemma}\label{lem:thick}
    Let $G=(V,E)$ be a thick spider and $(S,C,R)$ be the spider partition. Then, $meg(G)=\vert S\vert+\vert R\vert$, if $\vert C\vert=2$ and $meg(G)=\vert Man(G)\vert =\vert V\vert$, otherwise.
\end{lemma}

\begin{proof}
    Note that if $\vert C\vert=2$, then $G$ is a thin spider; hence, the result follows from the previous lemma. Now, let $\vert C\vert\geq 3$. Consider any two distinct vertices $c_i,c_j$ from $C$. The only induced $2$-path of type $c_jc_ix$ is $c_jc_is_j$, which is part of a $4$-cycle $c_jc_is_jc_k$, where $k\notin\{i,j\}$; such $k$ exists as $\vert C\vert\geq 3$. This implies that $c_i$ is a mandatory vertex for every $c_i\in C$. Now if $R\neq \emptyset$, then for any $r\in R$ (or $s\in S$), $N[r]\subset N[c]$ (or $N[s]\subset N[c]$), for any $c\in C$. Hence, by Observation \ref{observation:1}, $R\cup S$ is contained in every MEG set. Hence $meg(G)=\vert V\vert=\vert Man(G)\vert$.
\end{proof}

Now, the only case that remains to be seen is when $G$ is a join of two $P_4$-sparse graphs $G_1$ and $G_2$. Note that if $G_1$ and $G_2$ both are nontrivial graphs (having more than one vertex), then $meg(G)=\vert V(G)\vert$. Now assume that, without loss of generality, $\vert V(G_1)\vert=1$ and $V(G_1)=\{v\}$. Two cases can appear. In the first case, if $G_2$ is a disconnected graph, then $v$ is a cut-vertex and $N[x]\subseteq N[v]$ for any $x\in V(G_2)$, hence $meg(G)=\vert V(G_2)\vert$. For the other case, let $G_2$ be a connected graph. In that case, we propose the following lemma.

\begin{lemma}\label{lem:join}
    Let $G=G_1\oplus G_2$ and $V(G_1)=\{v\}$ and $G_2$ is a non trivial connected $P_4$-sparse graph. Then, 
    \begin{enumerate}
        \item $meg(G)=n$, if there exists a vertex $u$ in $G_2$, which satisfies the property that $d_{G_2}(u,x)\leq 2$ for every $x\in V(G_2)\setminus\{u\}$.
        \item $meg(G)=n-1$, otherwise.
    \end{enumerate}  
\end{lemma}

\begin{proof}
    Note that for every vertex $x\in V(G_2)$, $N[x]\subseteq N[v]$. Hence $x$ belongs to every MEG set of $G$, implying $V(G_2)\subseteq M$, where $M$ is any minimum MEG set of $G$. Suppose there exists a vertex $u\in V(G_2)$, with the property: $d_{G_2}(u,x)\leq 2$ for every $x\in X$. Now consider any induced $2$-path $uvy$. Clearly, $u,y$ are not adjacent. Hence $d_{G_2}(u,y)=2$. So, $uvy$ is part of a $4$-cycle, as $y$ is at distance $2$ from $u$ in the graph $G_2$. Hence, by Theorem \ref{th:mandatory}, $v$ is a mandatory vertex in $G$, for every $v\in V(G)$; which implies that $meg(G)=n$.
    
    Now, let us assume that no such vertex $u$ exists. Consider any edge $vx$, $x\in V(G_2)$. By the assumption there exists a vertex $x'\in V(G_2)$ such that $d_{G_2}(x,x')\geq 3$. Hence, the pair $x,x'$ monitors the edge $vx$. Since $x$ is arbitrary in $G_2$, then every edge of the form $vx$ can be monitored by some pair of vertices in $G_2$. Hence, $meg(G)=n-1$. 
\end{proof}

Hence, from the above discussion and Lemma \ref{lem:thin}, \ref{lem:thick}, \ref{lem:join} the following theorem can be concluded:

\begin{theorem}\label{th:p4}
    Given a connected $P_4$-sparse graph $G=(V,E)$, $meg(G)=\vert Man(G)\vert$ and thus, this set is a minimum MEG set of $G$, that can be computed in polynomial time.
\end{theorem}

\section{MEG for bipartite permutation graphs}
\label{sec:bip_per}
A \emph{strong ordering} $(<_X, <_Y)$ of a bipartite graph $G = (X\cup Y, E)$ entails an ordering $<_X$ for the set $X$ and an ordering $<_Y$ for the set $Y$. This ordering adheres to the condition that for any edges $ab$ and $a_0b_0$, where $a, a_0 \in X$ and $b, b_0 \in Y$, if $a$ precedes $a_0$ in $<_X$ and $b_0$ precedes $b$ in $<_Y$, then both $ab_0$ and $a_0b$ are edges in $G$. An ordering $<_X$ for $X$ exhibits the \emph{adjacency property} if, for each vertex in $Y$, its neighbours in $X$ appear consecutively in $<_X$. Moreover, the ordering $<_X$ is said to have the \emph{enclosure property} if, for any pair of vertices $y$ and $y_0$ in $Y$ where $N(y)$ is contained in $N(y_0)$, the vertices of $N(y_0) \setminus N(y)$ are consecutive in $<_X$. Strong ordering, adjacency property, and enclosure property, as delineated above, yield the subsequent result as discussed in previous literature.

\begin{theorem}[\cite{DBLP:journals/dam/SpinradBS87}]\label{th:bip_perm}
    The following statements are equivalent for a graph $G =(X\cup Y, E)$.
\begin{enumerate}
    \item $G = (X\cup Y, E)$ is a bipartite permutation graph.
    \item  $G$ has a strong ordering.
    \item There exists an ordering of X, which has the adjacency property and the
enclosure property.
\end{enumerate}

\end{theorem}

Hence, from Theorem \ref{th:bip_perm}, we can conclude that if $G$ is a bipartite permutation graph, then $G$ has a strong ordering. We state another result from the previous literature.

\begin{theorem}[\cite{DBLP:journals/siamdm/HeggernesHLN12}]
    Let $(<_X, <_Y )$ be a strong ordering of a connected bipartite
permutation graph $G = (X\cup Y, E)$. Then both $<_X$ and $<_Y$ have the adjacency property and the enclosure property.
\end{theorem}

Hence, given a bipartite graph $G=(X\cup Y,E)$, with its strong ordering $(<_X,<_Y)$, the following notations can be defined:
\begin{itemize}
    \item $f(v)=v'$, where $v'\in N(v)$ and $v'< u$ for every $u\in N(v)$
    \item $l(v)=v'$, where $v'\in N(v)$ and $u< v'$ for every $u\in N(v)$.
\end{itemize}

In the above definition of $f$ and $l$, the notation $<$ is used for both $<_X$ and $<_Y$. Note that the ordering $(<_X,<_Y)$ has the following properties (see \cite{LAI199733}):

\begin{enumerate}
    \item Given any vertex of $G$, its neighbor set consists of some consecutive vertices in $<_X$ or $<_Y$.
    \item For a pair of vertices $u, v$ from $X$ or $Y$, if $u < v$ then $f(u) \leq f(v)$ and $l(u) \leq l(v)$.
\end{enumerate}

Throughout this section, we assume that $G=(X\cup Y,E)$ is a bipartite permutation graph. Given a strong ordering $(<_X,<_Y)$ of $G$, let the vertices of $X$ and $Y$ be ordered as $x_1<_X x_2<_X\ldots x_p$ and $y_1<_Y y_2<_Y\ldots y_q$, where $p=\vert X\vert$ and $q=\vert Y\vert$. When the context is clear, we write $<$ instead of $<_X$ or $<_Y$. In this section we call $x_1,x_p,y_1,y_q$ as \emph{extreme} vertices and all the other vertices as \emph{interim} vertices.

\begin{lemma}\label{lem:bip_per}
    Given a connected bipartite permutation graph $G=(X\cup Y,E)$ with strong ordering $(<_X,<_Y)$ and an interim vertex $u\in V(G)$, $u$ is not a cut-vertex if and only if there exists $u_0\in N(u)$ such that $f(u_0)<u$ and $l(u_0)>u$.
\end{lemma}

\begin{proof}
    Without loss of generality, let $u\in X$ and $u=x_i$. Let there exists $u_0\in Y$ such that $f(u_0)<u$ and $l(u_0)>u$. For the sake of contradiction, assume that $u$ is a cut-vertex. Let $X_i=\{x_1,\ldots,x_{i-1}\}$ and $X_i'=\{x_{i+1},\ldots, x_p\}$. In the next claim, we show that both the graphs $G[X_i\cup N(X_i)]$ and $G[X_i'\cup N(X_i')]$ are connected. 
    \begin{claim}
         $G[X_i\cup N(X_i)]$ and $G[X_i'\cup N(X_i')]$ are connected. 
    \end{claim}

    \begin{proof}
        For the sake of contradiction, let there exists $x_j\in X_i$ such that $N(x_j)\cap N(x_{j+1})=\emptyset$. This implies that $l(x_j)<f(x_{j+1})$, hence $l(x_j)<f(x_{k})$ for all $k\in \{j+1,\ldots,p\}$. Hence, $N(x_j)\cap N(x_k)=\emptyset$, for all $k\in \{j+1,\ldots,p\}$. This implies that there does not exist any path from $x_j$ to $x_{j+1}$ in $G$, which contradicts the fact that $G$ is a connected graph. Hence, $N(x_j)\cap N(x_{j+1})\neq \emptyset$, for all $j\in \{1,\ldots,i-1\}$. This implies that $G[X_i\cup N(X_i)]$ is connected. Analogously, it can be shown that $G[X_i'\cup N(X_i')]$ is also connected.
    \end{proof}
    
    Since $u_0\in N(X_i)\cap N(X_i')$, $G[X_i\cup N(X_i)\cup X_i'\cup N(X_i')]$ is also connected (refer to Figure \ref{fig:bip_per} for better understanding). Hence $G\setminus \{u\}=G[(X_i\cup N(X_i)\cup X_i'\cup N(X_i'))]$ is a connected graph, which implies that $u$ is not a cut-vertex. If $u\in Y$, the proof follows from similar arguments.

    \begin{figure}[h!]
    \centering
    \begin{subfigure}{0.45\textwidth}
        \centering
         \includegraphics[scale=0.4]{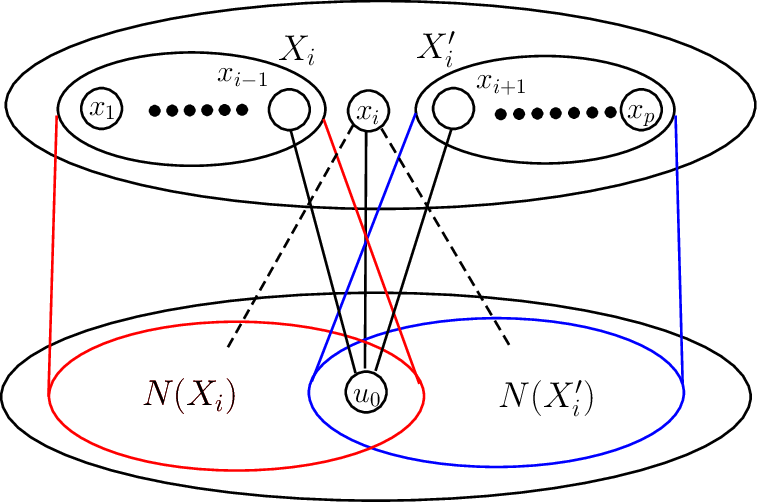}
        \caption{$u=x_i$ is not a cut-vertex}
        \label{fig:bip_per}
    \end{subfigure}
\hspace{0.5in}
    \begin{subfigure}{0.4\textwidth}
    \centering
         \includegraphics[scale=0.4]{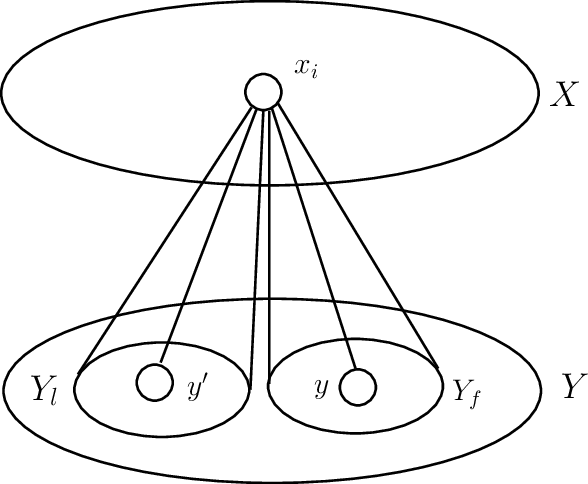}
        \caption{$u=x_i$ is not a cut-vertex}
        \label{fig:bip_per_2}
    \end{subfigure}
    \caption{Figures used in the proof of Lemma \ref{lem:bip_per}}
\label{pendant_2}
\end{figure}

    Conversely, assume that $u$ is not a cut-vertex. For the sake of contradiction, assume that for every $y_j\in N(x_i)$, either $f(y_j)=x_i$ or $l(y_j)=x_i$. Hence, $N(x_i)$ can be partitioned into two sets $Y_f=\{y\in N(x_i)~\vert~f(y)=x_i\}$ and $Y_l=\{y\in N(x_i)~\vert~l(y)=x_i\}$. Note that none of the set $Y_f$ and $Y_l$ can be empty as $x_i$ is an interim vertex and $G$ is a connected graph. Consider any two vertices, $y\in Y_f$ and $y'\in Y_l$. Note that every $y-y'$ path contains $x_i$ (refer to Figure \ref{fig:bip_per_2}). Hence, $G\setminus\{x_i\}$ is disconnected, which leads to a contradiction. Hence there exists a vertex $u_0\in N(u)$ such that $f(u_0)<u$ and $l(u_0)>u$.
 \end{proof}

    \begin{theorem}\label{lem:bip_per_2}
       Given a connected bipartite permutation graph $G = (X\cup Y, E)$, $Man(G) =V \setminus Cut(G)$ and thus, this set is a minimum MEG set of G, that can be computed in polynomial time.
    \end{theorem}

    \begin{proof}
        We show that every non-cut-vertex is a mandatory vertex. Note that if an extreme vertex $u\in V(G)$ is not a cut-vertex, then it is either a pendant vertex or there exists $u'$ such that $N(u)\subseteq N(u')$. In both cases, $u$ is a mandatory vertex.
        
        \smallskip
        \noindent Now, we consider an interim vertex $x_i\in X$ (without loss of generality), which is not a cut-vertex. By Lemma \ref{lem:bip_per}, there exists $y_j\in N(x_i)$ such that $f(y_j)<x_i$ and $l(y_j)>x_i$. Note that for any induced $2$-path $y_jx_iy_k$, two cases can occur.

        \smallskip
        \noindent\textbf{Case 1}: $y_k<y_j$. In this case $x_{i-1}\in N(y_k)\cap N(y_j)$. This implies that $y_jx_iy_k$ is part of a $4$-cycle $y_jx_iy_kx_{i-1}y_j$.

        \smallskip
        \noindent\textbf{Case 2}: $y_k>y_j$. In this case $x_{i+1}\in N(y_k)\cap N(y_j)$. This implies that $y_jx_iy_k$ is part of a $4$-cycle $y_jx_iy_kx_{i+1}y_j$.

        \smallskip
        \noindent Hence $x_i$ is a mandatory vertex, by Theorem \ref{th:mandatory}. The same can be shown for an interim vertex $y_i\in Y$ which is not a cut-vertex. Hence $Man(G)=V(G)\setminus Cut(G)$.
    \end{proof}

\section{MEG for strongly chordal graphs}
\label{sec:strong_chordal}
In this section, we consider that $G=(V,E)$ is a strongly chordal graph. We use the definition of strongly chordal graphs defined in \cite{DBLP:journals/dam/Farber84}. We define a vertex $v \in V$ as \emph{simple} if the vertices within its closed neighborhood can be arranged as follows: $N_G[v] = {v_1, v_2, \ldots , v_r}$, where $v_1 = v$ and $N_G[v_i] \subseteq N_G[v_j]$ for $1 \leq i \leq j \leq r$. A graph $G$ is strongly chordal if every induced subgraph of $G$ contains a simple vertex. An ordering $\alpha = (v_1, v_2,\ldots, v_n)$ of vertices in $V$ is called a strong elimination ordering (SEO) if $v_i$ is a simple vertex in $G[v_i,v_{i+1},\ldots,v_n]$. A graph is strongly chordal if it admits an SEO. We show that the set of mandatory vertices forms an MEG set for all strongly chordal graphs.

\begin{theorem}
	For any strongly chordal graph $G$, $Man(G)$ forms an optimal MEG set of $G$, and thus MEG can be solved in polynomial time for this class of graphs.
\end{theorem}

\begin{proof} 
	We show this by induction on the number of vertices in the graph, $n$. For $n=3$, the strongly chordal graphs are $P_3$ and $C_3$; the statement is true for both of these graphs. Hence, the statement is true for base cases. Now, let the statement be true for $n=k$. For $n=k+1$, let us consider a strongly chordal graph $G_{k+1}$ with SEO $v_1<v_2<\ldots<v_{k+1}$. Now define $G_k:=G_{k+1}[v_2,v_3,\ldots,v_{k+1}]$. Let us consider the neighbourhood of $v_1$ as $N(v_1)=\{v_{i_1},v_{i_2},\ldots,v_{i_l}\}$, where $N[v_{i_1}]\subseteq N[v_{i_2}]\subseteq\ldots \subseteq N[v_{i_l}]$. First, we show the following claims.
	
	\begin{claim}\label{Claim_5.1}
		$Man(G_{k})\setminus \{v_{i_l}\}\subseteq Man(G_{k+1})\subseteq Man(G_k)\cup\{v_1,v_{i_l}\}$.
	\end{claim}
	
	\begin{proof}
		Note that by Observation~\ref{observation:1},  $N[v_1]\setminus \{v_{i_l}\}\subseteq Man(G_{k+1})$ as for every $v\in N[v_1]\setminus \{v_{i_l}\}$, $N[v]\subseteq N[v_{i_l}]$. Now, $V(G_{k+1})$ can be partitioned as $V(G_{k+1})=N[v_1]\cup V'$, where $V'$ is the set of all non neighbours of $v_1$. Let us consider a vertex $v\in V'$. Two cases can appear.
		
	\smallskip
    \noindent \textbf{Case 5.1.1}: $v\in Man(G_k)$
	
    \smallskip
    \noindent This implies that a vertex $u\in N_{G_k}(v)$ exists such that every induced $2$-path $uvx$ is part of a $4$-cycle. Now, after adding $v_1$ to the vertex set, the neighbour set of $v$ has not expanded, as $v$ is not a neighbour of $v_1$. This implies that every induced $2$-path $uvx$ is still part of a $4$-cycle. This implies $v\in Man(G_{k+1})$. 
	
	\smallskip
    \noindent \textbf{Case 5.1.2}: $v\notin Man(G_k)$ 
	
    \smallskip
    \noindent For the sake of contradiction, let $v \in Man(G_{k+1})$. This implies that there exists $u\in N_{G_{k+1}}(v)$ such that every induced $2$-path $uvx$ is part of a $4$-cycle $uvxx'$. Note that $u\neq v_1$, as $v$ is not adjacent to $v_1$. But since $v$ was not a mandatory vertex in $G_k$, for this $u$, there must exist an induced $2$-path $uvx$ in $G_k$, which is not part of any $4$-cycle in $G_k$. This implies that the $4$-cycle $uvxx'$ in $G_{k+1}$, of which $uvx$ is part of, must include $v_1$. This implies that $x'=v_1$ and $u,x\in N(v_1)$. Hence, there is an edge between $u$ and $x$ (refer to Figure \ref{fig:claim1.1}), as the graph induced on $N[v_1]$ is a clique. So, $uvx$ is not an induced $2$-path, which leads to a contradiction. Hence,  $v \notin Man(G_{k+1})$.

      \begin{figure}[h!]
        \centering
        \includegraphics[scale=0.45]{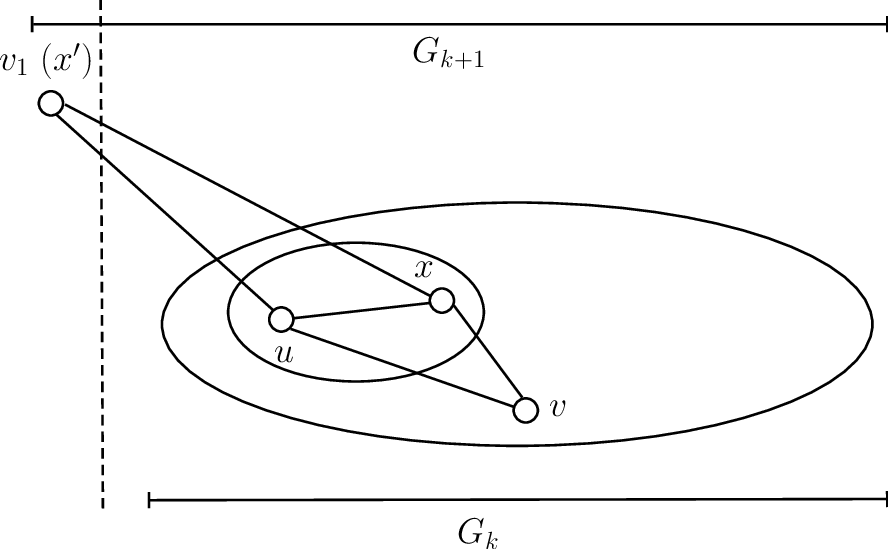}
        \caption{$v$ is a mandatory in $G_{k+1}$ but not in $G_k$}
        \label{fig:claim1.1}
    \end{figure}
	
	This implies $Man(G_k)\cap V'=Man(G_{k+1})\cap V'$ and $N(v_1)\setminus\{v_{i_l}\} \subseteq Man(G_k)$ and $N[v_1]\setminus\{v_{i_l}\} \subseteq Man(G_{k+1})$. Hence, the claim follows.
	\end{proof}

\begin{claim}
    If $v_1$ is a pendant vertex, then $Man(G_{k+1})$ is an optimal MEG set of $G$.
\end{claim}
\begin{proof}
    Note that if $v_1$ is a pendant vertex in $G_{k+1}$, then $N(v_1)=\{v_{i_1}\}$. This implies that $Man(G_{k+1})=(Man(G_k)\setminus\{v_{i_1}\})\cup \{v_1\}$. By the induction hypothesis, $Man(G_k)$ is an MEG set of $G_k$. If $v_{i_1}\in Man(G_k)$, then there exists an edge $uv \in E(G_k)$ that is only being monitored by $v_{i_1}$ and $v'$ for some $v'\in Man(G_k)$; such that no other pair in $Man(G_k)$ monitors $uv$. Evidently, $uv$ and $v_1v_{i_1}$ can both be monitored by $v_1,v'$ in $G_{k+1}$ (refer to Figure \ref{pendant_1}).

     \begin{figure}[h!]
    \centering
    \begin{subfigure}{0.4\textwidth}
        \centering
        \includegraphics[width=1\linewidth]{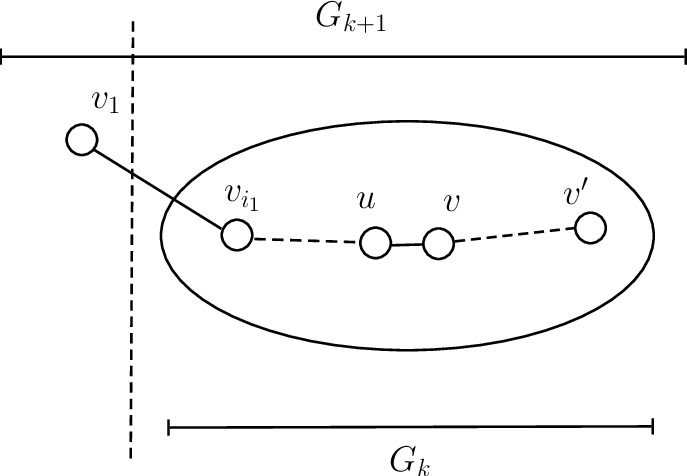}
    \caption{The pair $v_1,v'$ monitors both the edge $v_1v_{i_1}$ and $uv$}
    \label{pendant_1}
    \end{subfigure}
\hspace{0.5in}
    \begin{subfigure}{0.36\textwidth}
    \centering
        \includegraphics[width=1\linewidth]{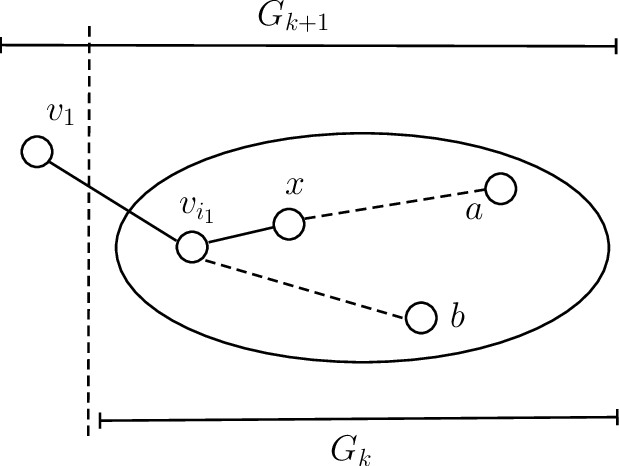}
    \caption{The pairs $v_1,a$ and $v_1,b$ both monitors the edge $v_{i_1}v_1$}
    \label{subfig:2}
    \end{subfigure}
    \caption{$v_1$ is a pendant vertex and $v_{i_1}$ is the stem}
\end{figure}

     If $v_{i_1}\notin Man(G_k)$, then all the edges in $G_{k+1}$ except $v_1v_{i_1}$ are being monitored by some pair of vertices in $Man(G_k)$. The only edge that remains to be monitored is $v_1v_{i_1}$. Consider any edge of the form $v_{i_1}x$ in $G_k$. It is monitored by some $a,b\in Man(G_k)$. Hence, it is not hard to see that either the pair $a,v_1$ or the pair $b,v_1$ monitors the edge $v_1v_{i_1}$ (refer to Figure \ref{subfig:2}). Hence in both of the cases, $meg(G_{k+1})=\vert Man(G_{k+1})\vert$.
\end{proof}
 
\medskip
\noindent So, from now on, let us assume that $\vert N_{G_{k+1}}(v_1)\vert\geq 2$, that is $v_1$ is not a pendant vertex in $G_{k+1}$. 

	\begin{claim}\label{Claim_5.2}
		If $v_{i_l}\in Man(G_{k+1})$, then $v_{i_l}\in Man(G_{k})$.
	\end{claim}
	
	\begin{proof}
	 For the sake of contradiction, let us assume that $v_{i_l}\in Man(G_{k+1})$ but $v_{i_l}\notin Man(G_{k})$. Then for every $u\in N_{G_k}(v_{i_l})$, there exists an induced $2$-path $uv_{i_l}x$ in $G_k$, which is not part of any $4$-cycle in $G_k$. But since, $v_{i_l}\in Man(G_{k+1})$, there exists $u_i\in N_{G_{k+1}}(v_{i_l})$ such that every induced $2$-path $u_iv_{i_l}x$ is part of a $4$-cycle $u_iv_{i_l}xx'$ in $G_{k+1}$. This $4$-cycle was clearly not present in $G_{k}$, which implies that either $u_i=v_1$ or $x'=v_1$. If $u_i=v_1$, then for every neighbour $x$ of $v_{i_l}$, $x$ is also adjacent to a vertex $x'$ in $N(v_1)\setminus \{v_{i_l}\}$ (refer to Figure \ref{fig:claim2.1}). This implies $x\in N[v_{i_{l-1}}]$, hence $N[v_{i_{l-1}}]= N[v_{i_l}]$, which implies that $v_{i_l}\in Man(G_k)$, by Theorem \ref{th:man_twin}, leading to a contradiction. If $x'=v_1$, then $u_i,x\in N(v_1)$. This implies that an edge exists from $u_i$ to $x$, as the graph induced on $N[v_1]$ is a clique, which contradicts the fact that $u_ivx$ is an induced $2$-path. Hence, $v_{i_l}\in Man(G_{k})$.
	\end{proof}

 \begin{figure}[h!]
        \centering
        \includegraphics[scale=0.45]{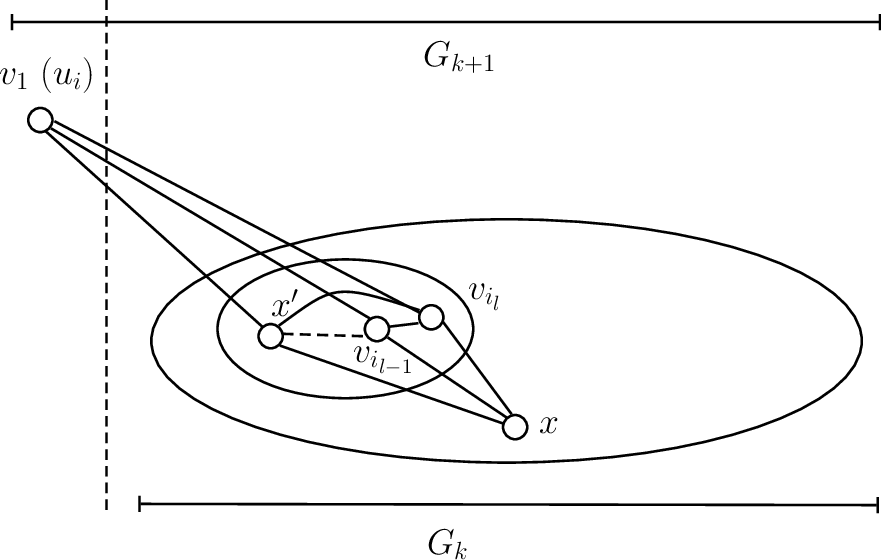}
        \caption{$u_i=v_1$ implies that $N[v_{i_{l-1}}]= N[v_{i_l}]$}
        \label{fig:claim2.1}
    \end{figure}
	
	Claim \ref{Claim_5.2} implies that if the vertex $v_{i_l}$ is mandatory in $G_{k+1}$, then it was also mandatory in $G_k$. Two cases may appear for $v_{i_l}$. The rest of the proof proceeds as follows:
    \begin{itemize}
        \item In Case 1, we consider $v_{i_l}\in Man(G_{k+1})$ and show that $Man(G_{k+1})$ forms an MEG set.
        \item In Case 2, we consider $v_{i_l}\notin Man(G_{k+1})$. Here again, two subcases may appear, $v_{i_l}\notin Man(G_{k})$ (Case 2.1) and $v_{i_l}\in Man(G_{k})$ (Case 2.2). In both of the subcases, we show that $Man(G_{k+1})$ forms an MEG set.
    \end{itemize}
 
	\noindent \textbf{Case 1}: $v_{i_l}\in Man(G_{k+1})$\\
	By Claim \ref{Claim_5.2}, $v_{i_l}\in Man(G_{k})$. In this case, $Man(G_{k+1})=Man(G_k)\cup\{v_1\}$. Now, by induction hypothesis, $Man(G_k)$ is an MEG set of $G_k$. Next we prove the following claim:
    \begin{claim}
        $v_1$ does not belong to any shortest $x-y$ path, where $x,y\in V(G_k)$.
    \end{claim}

    \begin{proof}
        Clearly, none of the $x,y$ is $v_1$ as they both belong to $V(G_k)$. For the sake of contradiction, let $P=((x=)x_1x_2\ldots x_k(=y))$ is a shortest $x-y$ path which contains $v_1$. Let $x_i=v_1$ for some $1<i<k$. Then the vertices $x_{i-1},~x_{i+1}$ are adjacent in $G_k$ (and in $G_{k+1}$ as well), as $N[v_1]$ induces a clique. This implies that $x_1x_2\ldots x_{i-1}x_{i+1}\ldots x_k$ is a path between $x$ and $y$ which has smaller length than $P$, which contradicts the fact that $P$ is a shortest $x-y$ path. Hence  $v_1$ does not belong to any shortest $x-y$ path.
    \end{proof}

    \medskip
    \noindent Hence, all the edges in $G_{k+1}$ that do not have $v_1$ as an endpoint are being monitored by some pair of vertices in $Man(G_k)$. Now since $N[v_1]\subseteq Man(G_{k+1})$, then all the edges of the form $v_1x$ are being monitored by $v_1$ and $x$, as $x\in N[v_1]\subseteq Man(G_{k+1})$. Hence, $Man(G_{k+1})$ is an MEG set of $G_{k+1}$.

    \medskip
    \noindent \textbf{Case 2}: $v_{i_l}\notin Man(G_{k+1})$\\
In this case, either $v_{i_l}\in Man(G_{k})$ or $v_{i_l}\notin Man(G_{k})$. 

\medskip
\noindent \textbf{Case 2.1}: $v_{i_l}\notin Man(G_{k})$

\smallskip
\noindent If $v_{i_l}\notin Man(G_{k})$, then all the edges in $E(G_{k+1})$ which do not have $v_1$ as an endpoint are monitored by some pair in $Man(G_k)$. All the edges of the form $v_1v_{i_j}$ are monitored by $v_1$, $v_{i_j}$, for every $1\leq j\leq l-1$. Now, the edge that is yet to be monitored is $v_1v_{i_l}$. Note that the edge  $v_{i_{l-1}}v_{i_l}$ is being monitored by $v_{i_{l-1}}$ and $v'$ for some $v'\in Man(G_k)$. This implies that every shortest path between $v_{i_{l-1}}$ and $v'$ contains the edge $v_{i_{l-1}}v_{i_l}$. Hence, every shortest $v_1-v'$ path also contains the edge $v_1v_{i_l}$, if not then there exists a shortest $v_{i_{l-1}}-v'$ path that bypasses the edge $v_{i_{l-1}}v_{i_l}$, which leads to a contradiction (refer to Figure \ref{fig:claim2.2}). Hence, $Man(G_{k+1})$ forms an MEG set of $G_{k+1}$.

 \begin{figure}[h!]
        \centering
        \includegraphics[scale=0.45]{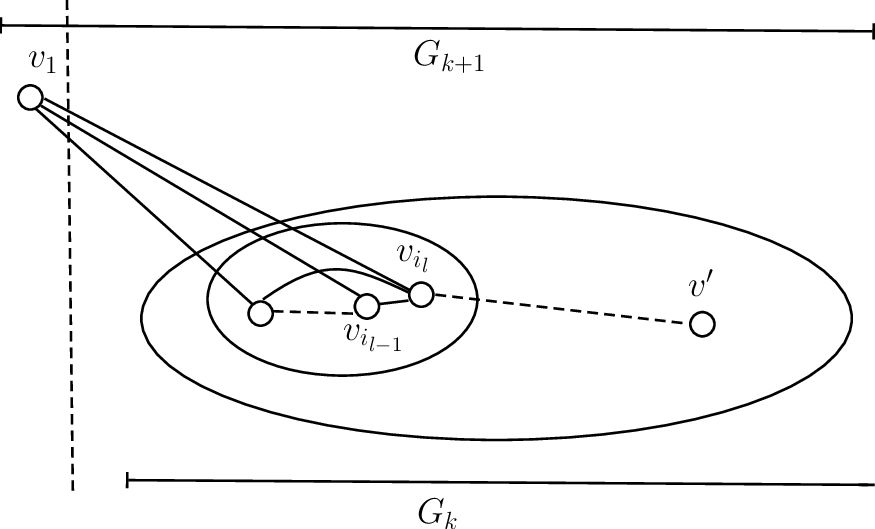}
        \caption{The pair $v_1, v'$ monitors the edge $v_1v_{i_1}$}
        \label{fig:claim2.2}
    \end{figure}
   
\noindent \textbf{Case 2.2}: $v_{i_l}\in Man(G_{k})$
    
   \smallskip
   \noindent Let $v_{i_l}\in Man(G_{k})$. Hence $Man(G_{k+1})=(Man(G_k)\setminus \{v_{i_l}\})\cup \{v_1\}$. Now define $E_l$ as the set of edges in $G_k$ which were only being monitored by $v_{i_l},v'$, for some $v'\in Man(G_k)$ and no other pairs in $Man(G_k)$. Note that every edge in $E(G_{k+1})\setminus E_l$ is being monitored by some pair in $Man(G_{k+1})$. Hence, it is sufficient to prove that every edge in $E_l$ can be monitored by some pair in $Man(G_{k+1})$. Consider an edge $e\in E_l$.

    \medskip
    \noindent \textbf{Case 2.2.1}: $v_{i_l}$ is not an endpoint of $e$.
    
    \smallskip
    \noindent If $v_{i_l}$ is not an endpoint of $e$, then in $G_k$, $e$ was monitored by $v_{i_l},v'\in Man(G_k)$. Note that $v'\notin N[v_1]$. Every shortest $v_{i_l}-v'$ path contains the edge $e$. Next we show that every shortest $v_1-v'$ path also contains $e$ (refer to Figure \ref{fig:e_non}).

    \begin{claim}
         Every shortest $v_1-v'$ path also contains $e$.
    \end{claim}
    
    \begin{proof}
        Let there exists a shortest $v_1-v'$ path $P=((v_1=)x_1x_2\ldots x_k(=v'))$ that does not contain $e$. Note that $v_{i_l}$ does not belong to the path (as it will contradict the fact that every shortest $v_{i_l}-v'$-path contains the edge $e$). Without loss of generality, let us assume that $x_2=v_{i_{l-1}}$. Note that $d(v_1,v')=1+d(v_{i_l},v')$, or $v_{i_l}x_2\ldots x_k(=v')$ is a shortest $v_{i_l}-v'$ path that does not contain $e$, leading to a contradiction. This implies that $d(v_{i_l},v')=d(v_{i_{l-1}},v')$. 
        
         \smallskip
    \noindent Hence, $v_{i_{l-1}}x_3\ldots x_k$ is a shortest $v_{i_{l-1}}-v'$ path that bypasses $e$. Since $N[v_{i_{l-1}}]\subseteq N[v_{i_l}]$, $x_3$ is also adjacent to $v_{i_l}$, hence $v_{i_l}x_3\ldots x_k$ is a shortest $v_{i_l}-v'$ path that avoids $e$, which is a contradiction to the fact that $v_{i_l},v'$ monitor $e$. Hence, every shortest $v_1-v'$ path also contains $e$.     
    \end{proof}

    \begin{figure}[h!]
        \centering
        \includegraphics[scale=0.45]{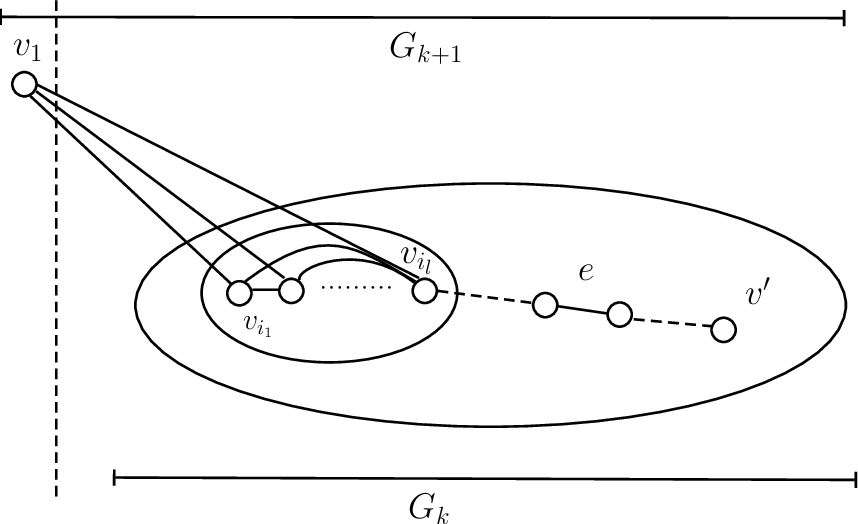}
        \caption{$e$ is not incident to $v_{i_l}$}
        \label{fig:e_non}
    \end{figure}

    \noindent \textbf{Case 2.2.2}: $v_{i_l}$ is an endpoint of $e$.
    
    \smallskip
    \noindent If $v_{i_l}$ is an endpoint of $e$, then let the other endpoint be $v$. Note that in $G_k$, the edge $v_{i_l}v$ was monitored by $v_{i_l}$ and $v'$ for some $v'$. 
    
    \smallskip
    \noindent At first let us consider the case $v\notin N_{G_{k+1}}(v_{i_{l-1}})$; such $v$ exists as $N[v_{i_{l-1}}]\subset N[v_{i_l}]$ (otherwise $v_{i_{l-1}},~v_{i_l}$ are twins, hence $v_{i_l}$ will be a mandatory vertex of $G_{k+1}$). Now note that $d_{G_{k+1}}(v,v_1)=2$ and $vv_{i_l}v_1$ is the unique shortest path from $v_1$ to $v$. Hence, the edge $v_{i_l}v$ is contained in every shortest $v_1-v'$ path (refer to Figure \ref{fig:claim_1_unotin}). Note that the edges $v_1v_{i_l}$ and $v_{i_{l-1}}v_{i_l}$ are also monitored by the pairs $v_1,v'$ and $v_{i_{l-1}},v'$, respectively.

    \begin{figure}[h!]
        \centering
        \includegraphics[scale= 0.45]{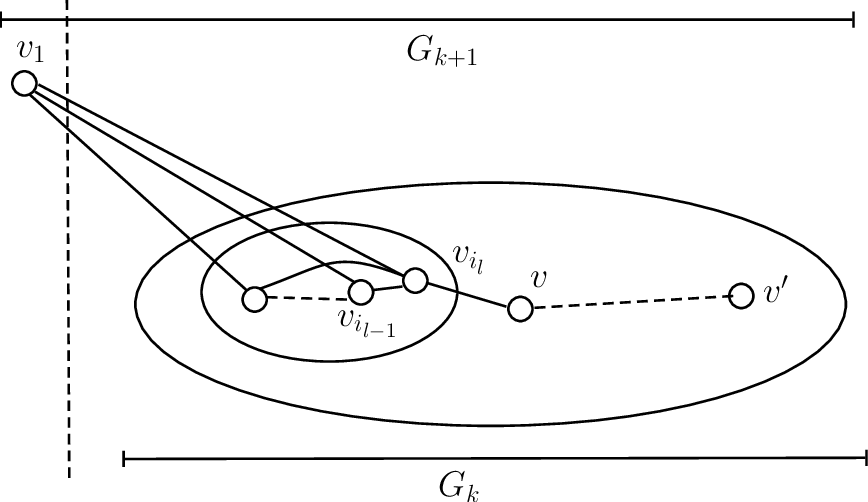}
        \caption{$v_{i_l}v$ is being monitored by $v_1$ and $v'$}
        \label{fig:claim_1_unotin}
    \end{figure}

    \smallskip
    \noindent Let $v\in N_{G_{k+1}}(v_{i_{l-1}})$. Note that there exists $x\in N[v_{i_{l}}]\setminus N[v_{i_{l-1}}]$ such that the induced $2$-path $vv_{i_l}x$ is part of no $4$-cycle. If not, then by Theorem \ref{th:mandatory} $v_{i_l}$ would be mandatory in $G_{k+1}$ as well. Now the edges $v_{i_l}v$ and $v_{i_l}x$ are being monitored by $v_{i_l},v'$ and $v_{i_l},x'$, respectively in $G_k$ (refer to Figure \ref{fig:claim_1_unotin_1}). We also consider $v'$ (and $x'$) to be the minimum distance mandatory vertex from $v_{i_l}$ such that $v_{i_l}$ and $v'$ ($v_{i_l}$ and $x'$) monitor the edge $v_{i_l}v$ ($v_{i_l}x$). Note that $v'\neq x'$; since if $v'=x'$, then every shortest $v_{i_l}-v'$ path contains both of the edges $v_{i_l}v$ and $v_{i_l}x$, which is a contradiction. 

      \begin{figure}[h!]
        \centering
        \includegraphics[scale=0.45]{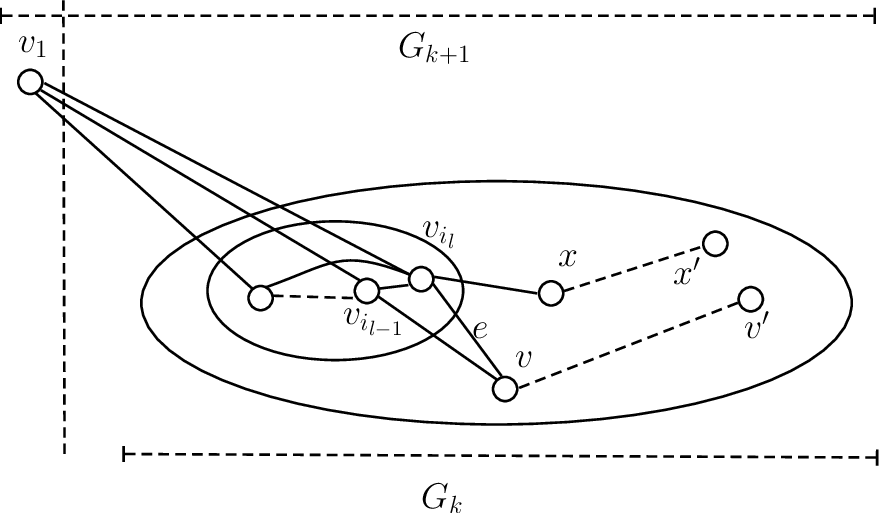}
        \caption{$e$ is incident to $v_{l_1}$}
        \label{fig:claim_1_unotin_1}
    \end{figure}

    \smallskip
    \noindent We show that $v',x'$ monitor the edge $v_{i_l}v$ in $G_{k+1}$. Now, for the sake of contradiction, assume that there exists a shortest $v'-x'$ path $P$ that bypasses the edge $v_{i_l}v$ (refer to Figure \ref{fig:main_proof_1}).

    \begin{figure}[h!]
    \centering
    \begin{subfigure}{0.41\textwidth}
        \centering
        \includegraphics[width=1\linewidth]{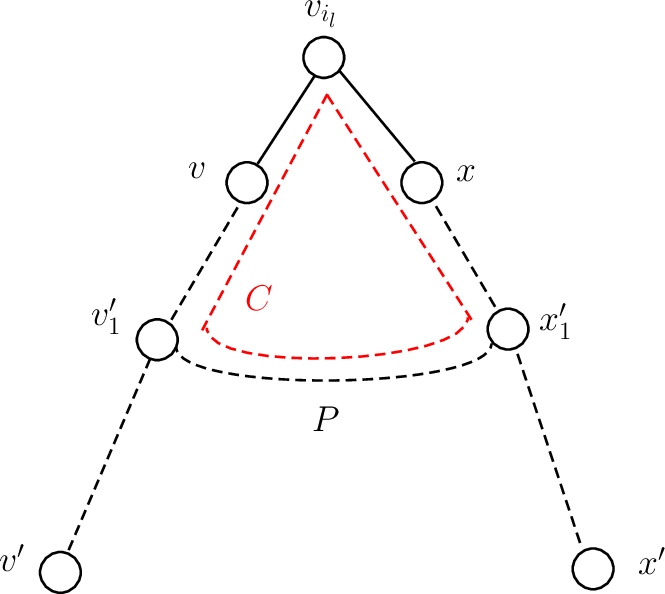}
    \caption{Another shortest path $P$ from $v'$ to $x'$}
    \label{fig:main_proof_1}
    \end{subfigure}
\hspace{0.5in}
    \begin{subfigure}{0.41\textwidth}
    \centering
        \includegraphics[width=1\linewidth]{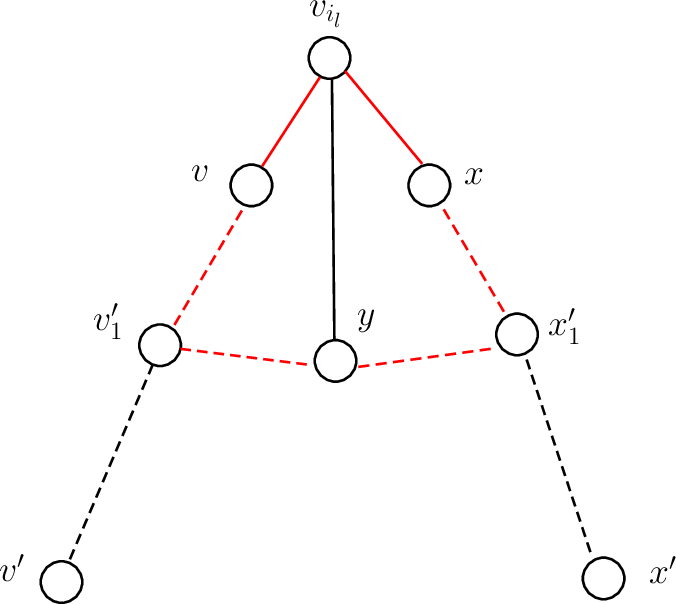}
    \caption{$v_{i_l}$ is incident to exactly one chord in $C$}
    \label{fig:main_proof_2}
    \end{subfigure}
    \caption{The cycle $C$ created by two distinct paths from $v'$ to $x'$}
\end{figure}

    \smallskip
    \noindent Let $v_1'$ (and $x_1'$) be the nearest vertex form $v_{i_l}$ that is present in $P\cap P_{v_{i_l}v'}$ (and $P\cap P_{v_{i_l}x'}$) (refer to Figure \ref{fig:main_proof_1}), where $P_{a,b}$ is defined to be any shortest $a-b$ path. Since $P$ is a shortest path from $v'$ to $x'$, then $d_P(v'_1,x'_1)\leq d(v'_1,v_{i_l})+d(x'_1,v_{i_l})$. Let $C$ be the cycle which is constructed with the boundaries $P(v'_1,x'_1)$, $P_{x'_1v_{i_l}}$ and $P_{v_{i_l}v'_1}$ (Refer to Figure \ref{fig:main_proof_1}). Note that there must be at least one chord attached to $v_{i_l}$ in the cycle $C$. If not, since $vv_{i_l}x$ form an induced 2-path, then $vv_{i_l}x$ is part of a cycle of length at least $4$, which does not have any chord, leading to a contradiction. Hence, $v_{i_l}$ has at least one chord attached to it in $C$. Let $v_{i_l}y$ be one such chord; then we show that $y$ lies in $P(v'_1,x'_1)$ in the next claim.

    \begin{claim}\label{claim:chord}
        $y$ lies in $P(v'_1,x'_1)$.
    \end{claim}

    \begin{proof}
        If $y$ does not lie in  $P(v'_1,x'_1)$, then $y$ must lie in the path between $v_{i_l}$ and $x'_1$ (or the path between $v_{i_l}$ and $v'_1$). If it lies on the path between $v_{i_l}$ and $x'_1$ (resp. between $v_{i_l}$ and $v'_1$), then we have a path from $v_{i_l}$ to $x'$  (resp. between $v_{i_l}$ and $v'$) which has length smaller than $d(v_{i_l},x')$ (resp. $d(v_{i_l},v')$), which is a contradiction. Hence,  $y$ lies in $P(v'_1,x'_1)$ (refer to Figure \ref{fig:main_proof_2}).
    \end{proof}
    
    \noindent Note that $d(x'_1,y)\geq d(x'_1,v_{i_l})$ and $d(v'_1,y)\geq d(v'_1,v_{i_l})$. If not, then there exists a shorter or equal length path from $v'$ to $v_{i_l}$ (from $x'$ to $v_{i_l}$) that bypasses the edge $v_{i_l}v$ (resp. $v_{i_l}x$), which is a contradiction to the fact that the pair $v_{i_l},v'$ (resp. $v_{i_l},x'$) monitors the edge $v_{i_l}v$ (resp. $v_{i_l}x$). Hence $d_P(x'_1,v'_1)=d(x'_1,y)+d(y,v'_1)\geq d(v'_1,v_{i_l})+d(x'_1,v_{i_l})$, implying $d_P(x'_1,v'_1)= d(v'_1,v_{i_l})+d(x'_1,v_{i_l})$. This implies that $d(x'_1,y)=d(x'_1,v_{i_l})$ and $d(v'_1,y)=d(v'_1,v_{i_l})$, which implies that the vertex $y$ is unique in $P(v'_1,x'_1)$. This fact together with Claim~\ref{claim:chord} implies that in $C$, $v_{i_l}$ is incident to exactly one chord. 

    \smallskip
    \noindent Now note that if $v$ and $y$ are not adjacent, then the induced $2$-path $yv_{i_l}v$ is part of a chordless cycle of length at least $4$, which is a contradiction. Hence $vy\in E(G_{k+1})$. Similarly, it can be shown that $xy\in E(G_{k+1})$. Hence $vv_{i_l}x$ is part of a $4$-cycle $vv_{i_l}xyv$, which leads to a contradiction. Hence $v'$ and $x'$ monitor $vv_{i_l}$.

    \smallskip
    \noindent Hence, until now, we showed that all edges in $E_l$ can be monitored by the vertices of $Man(G_k)\setminus \{v_{i_l}\}$. Hence, combining this fact with Claim \ref{Claim_5.1} and \ref{Claim_5.2}, we can conclude that $Man(G_{k+1})$ forms an optimal MEG set of $G_{k+1}$. Hence, $meg(G)=\vert Man(G)\vert$ for every strongly chordal graph $G$.  
\end{proof}

\section{Conclusion and future aspects}
\label{sec:conclusion}
In this paper, we solved the complexity status of the MIN-MEG problem for some well-known graph classes. Next, it will be interesting to address the following questions:

\begin{question}
    For any bipartite permutation graph $G$, we showed that $Man(G)=V\setminus Cut(G)$, and the graph mentioned in~\cite[Figure 4]{DBLP:journals/corr/abs-2403-09122} (a path on six vertices with an added universal vertex, which is a permutation graph) shows that this is not always the case for permutation graphs. However, do we have $|Man(G)|=meg(G)$ when $G$ is a permutation graph?
\end{question}

\begin{question}
    For any strongly chordal graph $G$, we showed that $|Man(G)|=meg(G)$, and this was also shown to be true for well-partitioned chordal graphs in~\cite{DBLP:journals/corr/abs-2403-09122,DBLP:conf/caldam/FoucaudMMSST24}. Does this property hold for chordal graphs as well?
\end{question}

\begin{question}
    As can be seen in Figure~\ref{fig:diagram}, the complexity status of the MIN-MEG problem is open for graph classes such as circle graphs, permutation graphs, chordal graphs, cocomparability graphs, etc. Is the MIN-MEG problem efficiently solvable for these graph classes?
\end{question}

\bibliographystyle{abbrv}
\bibliography{refer}

\end{document}